\documentclass[leqno,onethmnum]{siamltex}

\usepackage{color}
\usepackage{amssymb}
\usepackage{amsmath}
\usepackage{hyperref}

\newtheorem{example}{\it Example}
\newtheorem{remark}{\it Remark}

\newcommand{\C}{\mathcal{C}}%the symbol for an operator
\newcommand{\E}{\mathbb{E}}%the expectation

\title%[Dynamic behavior of stochastic gene expression models]
{Dynamic behavior of stochastic gene expression models in the presence of bursting}

\author{M.~C. Mackey\footnotemark[2]
\and M. Tyran-Kami\'nska\footnotemark[3] \and R. Yvinec\footnotemark[4]}

\begin{document}
\maketitle

\renewcommand{\thefootnote}{\fnsymbol{footnote}}

\footnotetext[2]{Departments of Physiology, Physics \& Mathematics, McGill University, 3655 Promenade Sir
William Osler, Montreal, QC, Canada, H3G 1Y6 ({\tt michael.mackey@mcgill.ca}). This research was supported by
Natural Sciences and Engineering Research Council of Canada.}

\footnotetext[3]{Institute of Mathematics, University of Silesia, Bankowa 14, 40-007 Katowice, Poland ({\tt
mtyran@us.edu.pl}). This research was supported by the Polish National Science Center grant no~N~N201~608240.}

\footnotetext[4]{Institut Camille Jordan, Universit\'{e} Claude Bernard Lyon 1, 43 boulevard du 11 novembre
1918, 69622 Villeurbanne cedex, France ({\tt yvinec@math.univ-lyon1.fr}). This research was supported by  the
Ecole Normale Superieure Lyon (ENS Lyon, France, RV).}

\renewcommand{\thefootnote}{\arabic{footnote}}

\pagestyle{myheadings} \thispagestyle{plain} \markboth{M.~C. MACKEY, M. TYRAN-KAMI\'NSKA, AND  R.
YVINEC}{DYNAMICS OF STOCHASTIC GENE EXPRESSION MODELS}
%\author{Michael C. Mackey}

%\address{Departments of Physiology, Physics \& Mathematics, McGill University,
%3655 Promenade Sir William Osler, Montreal, QC, CANADA, H3G 1Y6}
%
%\email{michael.mackey@mcgill.ca}

%\author{Marta Tyran-Kami\'nska}
%
%\address{Institute of Mathematics,
%University of Silesia, Bankowa 14, 40-007 Katowice, POLAND}
%
%\email{mtyran@us.edu.pl}

%\author{Romain Yvinec}
%
%\address{Institut Camille Jordan,
%             Universit\'{e}
%             Claude Bernard Lyon 1,
%               43 boulevard du 11 novembre 1918,
%               69622 Villeurbanne cedex,
%               FRANCE}
%
%\email{yvinec@math.univ-lyon1.fr}

%\keywords{invariant density, piecewise deterministic Markov process}

%\tableofcontents

\begin{abstract}
This paper considers the behavior of discrete and continuous mathematical models for gene expression in the
presence of transcriptional/translational bursting.  We treat this problem in generality with respect to the
distribution of the burst size as well as the frequency of bursting, and our results are applicable to both
inducible and repressible expression patterns in prokaryotes and eukaryotes.  We have given numerous examples of
the applicability of our results, especially in the experimentally observed situation that burst size is
geometrically or exponentially distributed.
\end{abstract}


\section{Introduction}\label{sec:intro}

Recent spectacular advances in the ability of experimentalists to monitor the temporal behavior of single molecules
\cite{Elf2007,Golding2005,Ozbudak2002,Raj2006,Raj2009,Suter2011,Xie2008} inside cells has led to a quantum leap in our knowledge of their behavior as well as a plethora of data that challenge mathematicians.  These techniques are so refined that they allow the single molecule quantification of the transcription of mRNA as well as the translation of the mRNA into protein.  This visualization has shown that in many cases these transcription and translation processes occur in quantal bursts in which a few molecules are produced during a discrete period of time.  An analysis of the data obtained from such experiments  has given us many details of the nature of the bursting kinetics that are being used to guide mathematical modeling of these fascinating processes.

This paper utilizes the two main approaches that have been employed to model these bursting processes, i.e. a discrete formulation for the numbers of molecules \cite{Shahrezaei2008} or a continuous one \cite{Friedman2006,Mackey2011} and illustrates the common features of both as well as the differences.
Modeling (as opposed to simulation \cite{Gillespie1977,McAdams1997} which we do {\it not} consider) of the details of gene expression as a discrete Markov process  has an extensive literature (c.f \cite{Hornos2005,Innocentini2007-09-01,Peccoud1995,Raj2006,Ramos2007,Ramos2010,Shahrezaei2008,Thattai2001}) that has recently seen a flurry of activity.   The other approach that has received extensive attention is modeling of the process as a continuous one  and \cite{Friedman2006,Lipniacki2006,Mackey2011,Ochab-Marcinek2010,Shahrezaei2008a} are representative of these efforts.  The reader can consult \cite{Kepler2001} for an excellent expository account of the connection between these two approaches.

In the discrete Markov models, steady-state analytical solutions of  the master equation can often be obtained using the moment generating function.  For the continuous model formulations, one needs to solve the Fokker-Planck-like  equations,
sometimes using Laplace transforms. When solutions are not
available, moment equations can be derived and usually solved
\cite{Paulsson2004,Tapaswi1987}. Though continuous models have many analytic advantages over discrete ones, it is also the case that
information of potential importance may be lost in the continuous model formulation compared with the discrete formulation.

This paper presents a general one dimensional model for bursting gene expression  in both a discrete Markov
process formulation as well as a continuous situation.  Section \ref{sec:notation} presents some general
background material while Section \ref{sec:discrete} presents the discrete version of the bursting model.
Section \ref{ssec:gen-discrete} develops the general formulation of the discrete model while Section
\ref{ssec:disc-geo} deals with the special case in which the burst amplitudes are geometrically distributed.
Section \ref{sec:cont} develops the corresponding continuous model of the bursting expression, with a general
development in Section \ref{ssec:gen-cont} and Section \ref{ssec:exp-cont} devoted to the situation where the
burst amplitudes are exponentially distributed--a situation often found experimentally. Section
\ref{ssec:separable} concludes with an examination of a generalization of the exponential distribution of burst
amplitudes. The paper ends with some general observations in  Section \ref{sec:conclusions}.  Throughout the
paper, our results are illustrated with numerous examples.

\section{Notation and background}\label{sec:notation}

Let the triple $(E,\mathcal{E},m)$ be a $\sigma$-finite measure space and let $L^1=L^1(E,\mathcal{E},m)$ with
norm denoted by $\|\cdot\|_1$. A linear operator $P$ on $L^1$ is called \emph{substochastic} (\emph{stochastic})
if $Pu\ge 0$ and $\|Pu\|_1\le \|u\|_1$ ($\|Pu\|_1= \|u\|_1$) for all $u\ge 0$, $u \in L^1$. We denote by $D$ the
set of all \emph{probability densities} on $E$, i.e.
$$
D=\{u\in L^1: \,\, u\ge 0,\,\, \|u\|_1=1\},
$$
so that a stochastic operator transforms a density into a density. In the particular case of a countable set $E$
with $\mathcal{E}$ being the family of all subsets of $E$ and $m$ the counting measure,
 the space $L^1$ will be denoted by $\ell^1$. % and  densities will be called \emph{probability mass functions}.

Let $\mathcal{P}\colon E\times \mathcal{E}\to[0,1]$ be a \emph{stochastic transition kernel}, i.e.
$\mathcal{P}(x,\cdot)$ is a probability measure for each $x\in E$ and the function $x\mapsto\mathcal{P}(x, B)$
is measurable for each $B\in\mathcal{E}$, and let $P$ be a stochastic operator on $L^1$. If
\begin{equation*}\label{eq:stkmo}
\int_B P u(x)m(dx)=\int_E \mathcal{P}(y,B)u(y)m(dy)\quad \text{for all } B\in \mathcal{E}, u\in D,
\end{equation*}
then $P$ is called the \emph{transition} operator corresponding to $\mathcal{P}$. A stochastic operator $P$ on
$L^1$ is called \emph{partially integral} or \emph{partially kernel} if there exists a measurable function
$p\colon E\times E\to[0,\infty)$ such that
$$
\int_E\int_E p(x,y)\,m(dy)\, m(dx) >0 \quad\text{and}\quad P u(x) \ge \int_E p(x,y)u(y)\, m(dy)
$$
for every density $u$. If, additionally,
\[
\int_E p(x,y)\,m(dx)=1,\quad  y\in E,
\]
then $P$ corresponds to the stochastic kernel
\[
\mathcal{P}(y,B)=\int_{B}p(x,y)\,m(dx),\quad y\in E, B\in \mathcal{E},\] and we simply say that $P$ has
\emph{kernel} $p$. Note that each stochastic operator on $\ell^1$ has a kernel.
% $[p(y,x)]_{x,y\in E}$ which is
%represented as a matrix and called a stochastic matrix in the theory of Markov chains.

We denote by $\mathcal{D}(A)$ the domain of a linear operator $A$. We say that $A\subseteq B$, or that $B$ is an
\emph{extension} of $A$, if $\mathcal{D}(A)\subseteq \mathcal{D}(B)$ and $B u=A u$ for $u\in \mathcal {D}(A)$.
The operator $A$ is said to be \emph{closable} if it has a closed extension.  If $A$ is closable, then the
closure $\overline{A}$ of $A$ is the minimal closed extension of $A$; more specifically, it is the closed
operator whose graph is the closure in $L^1\times L^1$ of the graph of $A$.  For an exposition of semigroup
theory we refer to \cite{engelnagel00}.

A  semigroup $\{P(t)\}_{t\ge0}$ of linear operators on $L^1$ is called \emph{substochastic} (\emph{stochastic})
if it is strongly continuous and for each $t>0$ the operator $P(t)$ is substochastic (stochastic). A density
$u^*$ is called \emph{invariant} or \emph{stationary} for $\{P(t)\}_{t\ge 0}$ if $u^*$ is a fixed point of each
operator $P(t)$, $P(t) u^*=u^*$ for every $t\ge 0$.

\begin{theorem}[{\cite[Theorem 2]{pichorrudnicki00}}]\label{thm:pr00}
Let  $\{P(t)\}_{t\ge 0}$ be a stochastic semigroup such that for some $t_0>0$ the operator $P(t_0)$ is partially
integral. If  the semigroup $\{P(t)\}_{t\ge 0}$ has only one invariant density $u^*$ and $u^*>0$ a.e. then
$$
\lim_{t\to\infty}\|P(t) u-u^*\|_1=0\;\;\text{for all}\;\;u\in D.
$$
\end{theorem}

%\section{The Bursting Model}

\section{A discrete bursting model formulated
as a Markov process}\label{sec:discrete}
This section considers bursting gene expression as a Markov process.

\subsection{The general case}\label{ssec:gen-discrete}
In this section we model the number of gene products as a pure-jump Markov process $X=\{X(t)\}_{t\ge 0}$ in the
state space $E=\{0,1,2,\ldots\}$. Thus a master equation governs the dynamics evolution of  probabilities. A
general one-dimensional bursting gene expression model \cite{Kepler2001} may be constructed as follows.  Let $n$
be the number of gene products and $P_n(t)=\Pr(X(t)=n)$ denote the probability of finding $n$ gene products
inside the cell at a given time $t$. We shall include a loss ($n \rightarrow n-1$) and gain ($n \rightarrow
n+k$) of functional processes in terms of the general rates $\gamma_n$ and $\lambda_n$, respectively. The step
size assumes the values $k=1, 2, \ldots$ and is a random variable (independent of the actual number of gene
products) with probability density function $h$, so that $\sum_{k=1}^{+\infty} h_k = 1.$ Therefore, our general
master equation describing the time evolution of the probabilities $P_n$ to have $n$ gene products in a cell is
an infinite set of differential equations
\begin{equation}
\label{eq:discret}
 \frac{dP_n}{dt}=\gamma_{n+1}P_{n+1}-\gamma_nP_n +
\displaystyle \sum_{k=1}^nh_k\lambda_{n-k}P_{n-k}-\lambda_nP_n, \quad n=0,1,\ldots,
\end{equation}
where we use the convention that $\sum_{k=1}^0=0$. We supplement~\eqref{eq:discret} with the initial condition
$P_n(0)=v_n$, $n=0,1,\ldots$, where $v=(v_n)_{n\ge 0}\in \ell^1$ is a probability density function of the
initial amount $X(0)$ of the gene product. In the following paragraphs, we consider the existence and uniqueness
of solutions of \eqref{eq:discret} together with convergence to a stationary distribution and then summarize our
results in Theorem \ref{t:exst}.

Assume that
\begin{equation}\label{a:con}
\lambda_0>0,\quad \gamma_0=0,\quad \gamma_n>0,\quad \lambda_n,h_n\ge 0,\quad n=1,2,\ldots, \quad
\sum_{n=1}^{+\infty} h_n = 1.
\end{equation}
The process $X$ is the minimal pure jump Markov process  with the jump rate function
$\varphi(n)=\lambda_n+\gamma_n, n\ge 0$, and the jump transition kernel $\mathcal{K}$ given by
\begin{equation}\label{d:opK}
\mathcal{K}(n,\{n+j\})=\left\{
               \begin{array}{ll}
                 q_n, & \text{if } j=-1, n\ge 1,  \\
                 (1-q_n) h_j, & \text{if } j\ge 1, n\ge 0, \\
                 0, & \text{otherwise},
               \end{array}
             \right.\quad q_n=\frac{\gamma_n}{\lambda_n+\gamma_n}.
\end{equation}
First, we recall the construction of $X$. Let $\{\xi_k\}_{k\ge 0}$, be a discrete time Markov chain in the state
space $E=\mathbb{Z}_+=\{0,1,\ldots\}$ with transition kernel $\mathcal{K}$ and let $\{\varepsilon_k\}_{k\ge 1}$
be a sequence of independent random variables, exponentially distributed with mean~$1$. Set $T_0=0$ and define
recursively the times of jumps of $X$ as
\[
T_k=T_{k-1}+\frac{\varepsilon_k}{\varphi(\xi_{k-1})},\quad k=1,2,\ldots.
\]
Starting from $X(0)=\xi_0$ we have
\[
X(t)=\xi_k,\quad T_k\le t<T_{k+1},\quad k=0,1,2,\ldots,
\]
so that the process is uniquely determined for all $t<T_\infty$, where  \[T_\infty=\lim_{k\to\infty} T_k\] is
called the explosion time. If the explosion time is finite, we can add the point $-1$ to the state space and we
can set $X(t)=-1$ for  $t\ge T_{\infty}$. The process $X$ is called \emph{nonexplosive}  if
$\mathbb{P}_i(T_\infty=\infty)=1$ for all $i\in E$, where $\mathbb{P}_i$ is the law of the process starting from
$X(0)=i$.
%We
%have $\mathbb{P}_i(t_\infty=\infty)=1$ if and only if (see e.g.
%\cite[Proposition12.19]{kallenberg02})
%\[
%\mathbb{P}_i\left(\sum_{k=0}^\infty
%\frac{1}{\varphi(\xi_{k})}=\infty\right)=1.\]
In particular, if the chain $\{\xi_k\}_{k\ge 0}$ is  recurrent, then $X$ is nonexplosive.

We now rewrite  equation \eqref{eq:discret} as an abstract Cauchy problem in the space $\ell^1$. We make use of
the results from~\cite{tyran09}. Let $K$ be the transition operator on $\ell^1$ corresponding to $\mathcal{K}$
defined as in~\eqref{d:opK}. For $v=(v_n)_{n\ge 0}\in \ell^1$ we have $(Kv)_0=q_1v_{1}$ and
\begin{equation*}
(Kv)_n=q_{n+1}v_{n+1}+\sum_{k=1}^n h_{k}(1-q_{n-k})v_{n-k}, \quad n=1,2,\ldots.
\end{equation*}
Define  the operator
\begin{equation*}
Gu=-\varphi u +K(\varphi u)\quad \text{for}\quad u\in \ell^1_\varphi=\{u\in\ell^1:\sum_{n=0}^\infty
\varphi_n|u_n|<\infty\}.
\end{equation*}
There is a substochastic semigroup $\{P(t)\}_{t\ge 0}$ on $\ell^1$ such that for each initial probability
density function $v\in \ell^1_\varphi$ the equation
\begin{equation}\label{e:dcp}
\frac{d u}{dt}=G(u), \quad t>0, \quad u(0)=v,
\end{equation}
has a nonnegative solution $u(t)$ which is given by $u(t)=P(t)v$ for $t\ge 0$ and
\[
(P(t)v)_n=\sum_{j=0}^\infty \mathbb{P}_j(X(t)=n,t<T_\infty)v_j,\quad n=0,1,\ldots.
\]
The process $X$ is nonexplosive  if and only if the semigroup $\{P(t)\}_{t\ge 0}$ is stochastic. Equivalently,
the generator of the semigroup $\{P(t)\}_{t\ge 0}$ is the closure of $(G,\ell^1_\varphi)$. In that case the
solution $u(t)$ of~\eqref{e:dcp} is unique and it is a probability density function for each $t$, if $v$ is has
these properties.
%In particular, if the operator $K$ has a strictly positive fixed point, then the semigroup $\{P(t)\}_{t\ge 0}$
%is stochastic. Thus, we now look for fixed points of $K$.

The equation for the steady state $p^*=(p^*_n)_{n\ge 0}$ of \eqref{eq:discret} is of the form
\begin{equation}\label{e:stdis}
\gamma_{n+1}p^*_{n+1}-\gamma_np^*_n + \displaystyle \sum_{k=1}^nh_k\lambda_{n-k}p^*_{n-k}-\lambda_np^*_n=0,
\quad n=0,1,\ldots.
\end{equation}
Observe that $\gamma_1p^*_{1}=\lambda_0 p^*_{0}$ and that we can rewrite~\eqref{e:stdis} as
\[
\gamma_{n+1}p^*_{n+1}-\gamma_np^*_n=\lambda_np^*_n-\sum_{k=0}^{n-1}h_{n-k}\lambda_{k}p^*_{k}, \quad n=1,2\ldots.
\]
Hence
\[
\gamma_{n+1}p^*_{n+1}= \sum_{j=0}^n\lambda_jp^*_j-\sum_{j=1}^n\sum_{k=0}^{j-1}h_{j-k}\lambda_{k}p^*_{k}
\]
 and changing the order of
summation, we obtain
\begin{equation}\label{e:stdist}
p^*_{n+1}=\frac{1}{\gamma_{n+1}}\sum_{k=0}^{n}\overline{h}_{n-k}\lambda_k p^*_k,\quad n=0,1,\ldots,
\end{equation}
where
\[
\overline{h}_{l}=\sum_{j=l+1}^\infty h_j, \quad l\ge 0. \]
 Thus given $p^*_{0}$, equation~\eqref{e:stdist} uniquely determines $p^*$. Consequently, there is one, and
up to a multiplicative constant only one, solution of equation~\eqref{e:stdis}. If $p^*_{0}>0$ and either
$\overline{h}_l>0$ for all $l\ge 0$ or $\lambda_l>0$ for all $l\ge 1$, then $p^*_{n}>0 $ for all $n\ge 1$. Now,
if
\begin{equation}\label{e:exst}
\sum_{n=0}^\infty p^*_n=1\quad \text{and}\quad \sum_{n=0}^\infty (\lambda_n+\gamma_n)p^*_n<\infty,
\end{equation}
then  $p^*\in \ell^1_\varphi$, $G(p^*)=0$,  and $K(\varphi p^*)=\varphi p^*$, which implies that the semigroup
$\{P(t)\}_{t\ge 0}$ is stochastic. We have thus proved the following result, which is an analog of Theorem
\ref{thm:pr00} for the discrete bursting model.

\begin{theorem}\label{t:exst} Assume condition~\eqref{a:con} and
suppose that a strictly positive $p^*=(p^*_n)_{n\ge 0}$ given by \eqref{e:stdist} satisfies~\eqref{e:exst}.
 Then for each initial probability density function
$v=(v_n)_{n\ge 0}\in \ell^1_\varphi$ equation~\eqref{eq:discret} has a unique solution which is a probability
density function for each $t>0$ and satisfies
\[
\lim_{t\to\infty}\sum_{n=0}^\infty |(P(t)v)_n-p_n^*|=0.
\]
\end{theorem}

\begin{remark}\label{r:mean}\em
 From~\eqref{e:stdist} it follows that
\[
\sum_{n=0}^\infty \gamma_{n+1}p^*_{n+1}=\sum_{n=0}^\infty\sum_{k=0}^{n}\left(\sum_{j=n-k+1}^\infty
h_j\right)\lambda_k p^*_k=\sum_{k=0}^\infty\sum_{n=k}^{\infty}\left(\sum_{j=n-k+1}^\infty h_j\right)\lambda_k
p^*_k.
\]
The mean value $\E(h)$ of the distribution $h$ can be represented as
\[
\E(h)=\sum_{j=0}^\infty jh_j=\sum_{n=0}^{\infty}\sum_{j=n+1}^\infty h_j.
\]
We thus obtain
\[
\E(h)=\sum_{n=k}^{\infty}\sum_{j=n-k+1}^\infty h_j
\]
for each $k\ge 0$. Combining these leads to
\[
\sum_{n=0}^\infty \gamma_{n+1}p^*_{n+1}=\E(h)\sum_{k=0}^\infty \lambda_k p^*_k.
\]
\end{remark}

%\begin{prop}
%Assume that condition~\eqref{a:con} holds and that
%%\[ \E(h)=\sum_{j=0}^\infty jh_j<\infty.\] If
%\begin{equation}
%\limsup_{n\to\infty}\frac{\lambda_n\E(h)}{\gamma_n}<1\quad \text{and}\quad\liminf_{n\to\infty}\gamma_n>0.
%\end{equation}
%Then $p^*=(p^*_n)_{n\ge 0}$ given by \eqref{e:stdist} satisfies~\eqref{e:exst}.
%\end{prop}

\subsection{Bursting with a geometric distribution}\label{ssec:disc-geo}
Next, we give  sufficient conditions for~\eqref{e:exst} in the case when $h$ is geometric
\begin{equation}\label{e:geom}
h_k=(1-b)b^{k-1},\quad k=1,2,\ldots,
\end{equation}
with $b\in(0,1)$. Since
\[
\sum_{j=n-k+1}^\infty h_j=b^{n-k},
\]
we obtain the following equation for $p^*=(p^*_n)_{n\ge 0}$
\begin{equation}\label{eq:rec_bif}
\frac{p^*_{n+1}}{p^*_n}=\frac{\lambda_n+b\gamma_n}{\gamma_{n+1}},\quad n=0,1\ldots.
\end{equation}
Explicit stationary solutions in this case were recently obtained in \cite{Aquino11}. However, for $h$
geometric, we can go further and prove convergence to this stationary state with the following result which
follows from Theorem~\ref{t:exst} and Remark~\ref{r:mean}.

\begin{corollary} Assume that condition~\eqref{a:con} holds. Suppose that $h$ is geometric as in \eqref{e:geom}.
Then $p^{*}=(p^*_n)_{n\ge 0}$ is given by
\begin{equation}\label{e:stwheng}
p^*_n=p^*_0\prod_{k=1}^n\frac{\lambda_{k-1}+b\gamma_{k-1}}{\gamma_k}
=\frac{p^*_0\lambda_0}{\gamma_n}\prod_{k=1}^{n-1}\frac{\lambda_{k}+b\gamma_{k}}{\gamma_k},\quad n=1,2,\ldots.
\end{equation}
In particular, if
\begin{equation}\label{e:scge}
\limsup_{n\to\infty}\frac{\lambda_n}{\gamma_n}<1-b\quad \text{and}\quad \liminf_{n\to\infty} \gamma_n>0,
%\lim_{n\to\infty}\frac{\gamma_n}{\gamma_{n+1}}=1,
\end{equation}
then the conclusions of Theorem~\ref{t:exst} hold.
\end{corollary}

\begin{example}\em Consider $\lambda_n$ to be  a Hill function of the form
\begin{equation}
\label{eq:nonlinear}
\lambda_n = \lambda \dfrac{1+\Theta n^N}{\Lambda + \Delta n^N}
 %\lambda_n=\lambda \frac{1+K_1n^N}{K_0 + K_1 n^N},
\end{equation}
where  $\Lambda,\Delta, N>0$ and $\Theta\ge 0$.   %$K_1,K_0,\lambda>0$ and $N\ge 1$.
If $h$ is geometric and
\[
\liminf_{n\to\infty}\gamma_n>\frac{\lambda \Theta}{\Delta(1-b)},
\]
then condition \eqref{e:scge} holds.
%then there always exists $p^{*}=(p^*_n)_{n\ge 0}$ satisfying~\eqref{e:stdist}, \eqref{e:exst}. Thus the
%conclusions of Theorem~\ref{t:exst} hold.
%\[
%\lim_{n\to\infty}\gamma_n=\infty, \quad \lim_{n\to\infty}\frac{\gamma_n}{\gamma_{n+1}}=1,
%\]
%then irrespective of $b$ there always exists $p^{*}=(p^*_n)_{n\ge 0}$ satisfying~\eqref{e:stdist}.
\end{example}

%%%%%%%%%%%%%%%
\begin{remark}[Bifurcation in the discrete case]\label{rem:bif_disc}\em
Equation \ref{eq:rec_bif} can be used to examine the bifurcations in the stationary density, defined as changes in the number of maxima,
as a function of the model parameters. The number of maxima are linked to the number of sign changes of
\begin{equation}\label{eq:sign_fct}
n\mapsto \lambda_n+b\gamma_n-\gamma_{n+1}.
\end{equation}
In particular, $p^*$ has a maximum at $0$ if $ \lambda_0<\gamma_1$, and each successive sign change of
\eqref{eq:sign_fct} gives a maximum/minimum of $p^*$.
\end{remark}
%%%%%%%%%%%%%%%

We now provide examples for which the stationary distribution can be identified explicitly. In the following
examples we assume that $h$ is geometric with parameter $b$ as in \eqref{e:geom} and that $\gamma_n=\gamma n$,
$n\ge 1$, with $\gamma>0$.

\begin{example}[Negative binomial]\label{ex:nb}\em
Suppose that $\lambda_n=\lambda_0+\lambda n$ with $\lambda_0>0, \lambda\ge 0$. We have $\lambda_n\ge 0$ for each
$n$. Substituting $\gamma_k$ and $\lambda_k$ into \eqref{e:stwheng} gives
\begin{equation*}
p^*_n=\frac{p^*_0}{n!} \prod_{k=0}^{n-1}\left(\frac{\lambda_0}{b\gamma+\lambda}+k\right)
\left(\frac{\lambda+b\gamma}{\gamma}\right)^n,\quad n=0,1,\ldots.
\end{equation*}
Thus $p^*\in \ell^1$ if and only if
\[
\lambda +b\gamma <\gamma.
\]
In that case we obtain the negative binomial distribution
\begin{equation*}
\label{eq.negativebinom} p^*_n=\frac{(a)_n}{n!}p^n(1-p)^a, \quad n=0,1,\ldots,
\end{equation*} where
\begin{equation*}
p=\frac{\lambda+b\gamma}{\gamma},\quad a=\frac{\lambda_0}{b\gamma+\lambda},
\end{equation*}
and  $(a)_n$ is the Pochhammer symbol defined by
$$(a)_n=\frac{\Gamma(a+n)}{\Gamma(a)}=a(a+1)(a+2)\ldots(a+n-1), \quad(a)_0=1.$$
This was previously obtained in \cite{Shahrezaei2008}.
\end{example}

\begin{example}[Mixture of logarithmic distribution]\em
Suppose that $\lambda_0>0$ and $\lambda_n=0$ for $n\ge 1$. Then
\[
p^*_n=p^*_0\frac{\lambda_0}{\gamma}\frac{b^{n-1}}{n}, \quad
n=1,2,\ldots, % \quad p^*_0=\frac{b\gamma}{b\gamma-\lambda_0
%\ln(1-b)}.
\]
which can be rewritten as
\[
p^*_n=-\frac{b^{n}}{n\ln(1-b)}(1-p^*_0),\quad n=1,2,\ldots,\quad p^*_0=\frac{b\gamma}{b\gamma-\lambda_0
\ln(1-b)}.
\]
The distribution
\[
\tilde{p}_0=0,\quad \tilde{p}_n=-\frac{b^n}{n\ln(1-b)},\quad n=1,2,\ldots,
\]
is called a logarithmic distribution.

If we assume that $\lambda_n=0$ for $n> m$, then we obtain the following distribution
\[
p^*_n=p^*_0\frac{b^n}{n!}\prod_{k=0}^{n-1}\left(\frac{\lambda_k}{b\gamma}+k\right),\quad n=0,\ldots,m,
\]
and
\[
p^*_n= \frac{b^n}{cn} \left (1-\sum_{j=0}^mp^*_j \right ),\quad n> m,
\]
where $c$ and $p^*_0$ are such that
\[
c=\sum_{j=m+1}^\infty \frac{b^j}{j}\quad\text{and}\quad \sum_{j=0}^mp^*_j+p^*_m\frac{mc}{b^m}=1.
\]
In particular, this type of distribution will be obtained if we take $\lambda_0>0$, $\lambda<0$, and
\[
\lambda_n=\left\{
            \begin{array}{ll}
              \lambda_0+\lambda n, & \text{if } n\le -\lambda_0/\lambda,\\
              0, & \text{otherwise}.
            \end{array}
          \right.
\]
\end{example}

\begin{example}[Hypergeometric distributions]\label{ex:Gauss}\em
We now take
\begin{equation}
\label{eq:nonlinearN1} \lambda_n = \lambda \dfrac{1+\Theta n}{\Lambda + \Delta n}
% \lambda_n=\lambda \frac{1+K_1n}{K_0 + K_1 n},\quad n=0,1,\ldots,
\end{equation}
where $\lambda>0, \Lambda\ge 1, \Theta\ge \Delta$. We find that, for each $n$,
\[
\frac{\lambda_n+b\gamma n}{\gamma}=\frac{b(n+a_1)(n+a_2)}{n+b_1},
\]
where
\[
b_1=\frac{\Lambda}{\Delta},\quad a_1=\frac{1}{2}\left(\alpha-\beta\right),\quad
a_2=\frac{1}{2}\left(\alpha+\beta\right),
\]
and
\[
\alpha=\frac{\Lambda}{\Delta}+\frac{\lambda\Theta}{b\gamma\Delta},\quad \beta^2= \alpha^2-\frac{4\lambda}{
b\gamma\Delta}.
\]
 Since $\Lambda\ge 1$ and $\Theta\ge \Delta$, we can find a nonnegative $\beta$, thus $a_{2}\ge a_{1}> 0$.
 Consequently, the stationary distribution
is of the form
\begin{equation}
p^*_n=\frac{1}{{}_2F_1(a_1,a_2;b_1;b)}\frac{(a_1)_n(a_2)_n}{(b_1)_n }\frac{ b^n}{n!}, \quad n=0,1,\ldots,
\end{equation}
where ${}_2F_1$ is  Gauss' hypergeometric function
\[
{}_2F_1(a_1,a_2;b_1;x)=\sum_{n=0}^\infty \frac{(a_1)_n (a_2)_n}{(b_1)_n}\frac{x^n}{n!}.
\]
\end{example}

\begin{example}[Generalized hypergeometric distributions]\em
The generalized hypergeometric function ${}_pF_q$  is defined to be the real analytic  function on
$\mathbb{R}$ given by the
 series expansion
\[
{}_pF_q(a_1,\ldots,a_p;b_1,\ldots, b_q;x)=\sum_{n=0}^\infty \frac{(a_1)_n\ldots (a_p)_n}{(b_1)_n\ldots
(b_q)_n}\frac{x^n}{n!}.
\]
The negative binomial distribution in Example~\ref{ex:nb} for the case of $\lambda=0$ has the probability
generating function $s\mapsto {}_1F_0(a_1;b s)/{}_1F_0(a_1;b) $ with $a_1=\lambda_0/b\gamma$. The distribution
obtained in Example~\ref{ex:Gauss} has the probability generating function $$s\mapsto
\frac{{}_2F_1(a_1,a_2;b_1;b s)}{{}_2F_1(a_1,a_2;b_1;b)}. $$ Extending both of these examples we suppose that
$\lambda_n\ge 0$ is a rational function of $n$ satisfying
\[
\frac{\lambda_n+b \gamma n}{\gamma}=\frac{(n+a_1)\ldots(n+a_{q+1})b}{(n+b_1)\ldots(n+b_q)},\quad n=0,1,2,\ldots.
\]
Then  $p^*=(p^*_n)_{n\ge 0}$  has  the probability generating function
\[
\frac{{}_{q+1}F_q(a_1,\ldots,a_{q+1};b_1,\ldots, b_q; b s)}{{}_{q+1}F_q(a_1,\ldots,a_{q+1};b_1,\ldots, b_q;b)}.
\]
\end{example}

\section{Continuous bursting model}\label{sec:cont}

\subsection{The general case}\label{ssec:gen-cont}
In this section we consider a continuous state space version of the model presented in
Section~\ref{sec:discrete}, which is a piecewise deterministic Markov process (PDMP) $Y=\{Y(t)\}_{t\ge 0}$ with
values in $E=(0,\infty)$ where $Y(t)$ denotes the amount of the gene product in a cell at time $t$, $t\ge 0$. We
assume that protein molecules undergo degradation at a rate $\gamma$ that is interrupted by production at random
times
\[
t_1<t_2<\ldots
\]
occurring with intensity $\varphi$, and that both $\varphi$ and $\gamma$ depend on the current number of
molecules.
At each $t_k$ a random amount of protein molecules is produced, %independent of the current number of proteins,
so that the process changes from $Y(t_k-)$ to $Y(t_k)=Y(t_k-)+e_k$, $k=1,2,\ldots$, where $\{e_k\}_{k\ge 1}$ is
a sequence of random variables such that
\[
\Pr(e_k\in B|Y(t_k-)=y)=\int_{B}h(x,y)dx,
\]
where $h$ is a  nonnegative measurable function satisfying
\begin{equation}\label{eq:h}
\int_{0}^\infty h(x,y)dx=1,\quad y>0.
\end{equation}
The time-dependent probability density function $u(t,x)$ is described by the continuous analog of the master
equation \cite{mackeytyran08,Mackey2011}
\begin{equation}
\label{eq:continuous}
 \dfrac{\partial u(t,x)}{\partial
t}=\dfrac{\partial ( \gamma(x)u(t,x))}{\partial x}-\varphi(x)u(t,x)+ \int_0^x\varphi(y)u(t,y)h(x-y,y)dy
\end{equation}
with the initial probability density $ u(0,x)=v(x),$ $x>0$.

 We assume that $\gamma$ is a continuous function such that
\begin{equation}\label{e:as}
\gamma(x)>0 \quad \text{for }x>0, \quad \int_{0}^{\delta}\frac{dx}{\gamma(x)}=+\infty,
\end{equation}
for some $\delta>0$ and that $\varphi$ is a nonnegative  measurable function with $\varphi/\gamma$ being locally
integrable on $(0,\infty)$ and satisfying
\begin{equation}\label{e:as1}
\int_{0}^{\delta}\frac{\varphi(x)}{\gamma(x)}dx=+\infty.
\end{equation}
 From \eqref{e:as} it follows that the differential equation
\begin{equation}
x'(t)=-\gamma(x(t)),\quad x(0)=x>0,
\end{equation}
has a unique solution which we denote by $\pi_tx$, $t\ge 0$, $x>0$. For each $x>0$ we have $\pi_t x>0$ for all
$t>0$ and $\pi_t x\to 0$ as $t\to\infty$. This and condition~\eqref{e:as1} give
\[
\int_{0}^t \varphi(\pi_sx)ds=\int_{\pi_t x}^{x}\frac{\varphi(y)}{\gamma(y)}dy\to\infty,\quad \text{as
}t\to\infty,
\]
which implies that the function
\begin{equation*}
t\mapsto 1-e^{-\int_0^t \varphi(\pi_sx)ds}
\end{equation*}
is a distribution function of a positive and finite random variable for every $x>0$.

We now recall the construction of the minimal piecewise deterministic Markov process~$Y$ (see
e.g.~\cite{davis84,davis93} or \cite{tyran09} for details). Let $\{\varepsilon_k\}_{k\ge 1}$ be a sequence of
independent random variables exponentially distributed with mean $1$, which is also independent of
$\{e_k\}_{k\ge 1}$.  Set $t_0=0$. For each $k=1,2,\ldots$ and given $Y(t_{k-1})$ the process evolves as
\begin{equation}\label{e:bj}
Y(t)=\left\{
      \begin{array}{ll}
        \pi_{t-t_{k-1}}Y(t_{k-1}), & t_{k-1}\le t< t_{k}, \\
       Y(t_{k-})+e_k , &  t=t_k,
      \end{array}
    \right.
\end{equation}
where $t_k=t_{k-1}+\Delta t_k$ and $\Delta t_k$ is a random variable such that
\begin{equation*}
\Pr(\Delta t_k\le t|Y(t_{k-1})=x)=1-e^{-\int_0^t\varphi(\pi_sx)ds},\quad t, x>0.
\end{equation*}
The random variable $\Delta t_k$ can be defined with the help of the exponentially distributed random variable
$\varepsilon_k$ through the equality in distribution
\begin{equation*}\label{e:jrf}
\varepsilon_k=\int_{0}^{\Delta t_k} \varphi(\pi_{s} Y(t_{k-1}))ds,
\end{equation*}
which can be rewritten as
\[
\varepsilon_k=Q(\pi_{\Delta t_k}Y(t_{k-1}))-Q(Y(t_{k-1})),
\]
where the non-increasing function $Q$ is given by
\begin{equation}
Q(x)=\int^{\bar{x}}_x \frac{\varphi(y)}{\gamma(y)}dy,
\end{equation}
and $\bar{x}=+\infty$, when the integral is finite or any $\bar{x}>0$ otherwise. Since $Y(t_{k}-)=\pi_{\Delta
t_k}Y(t_{k-1})$, we obtain the following stochastic recurrence equation for  $\{Y(t_k)\}_{k\ge 0}$
\begin{equation}\label{e:dcmj}
Y(t_{k})=Q^{-1}(Q(Y(t_{k-1}))+\varepsilon_{k})+e_{k}, \quad k=1,2,\ldots,
\end{equation}
where $Q^{-1}$ is the generalized inverse of $Q$, $Q^{-1}(r)=\sup\{x:Q(x)\ge r\}$. Consequently, $Y(t)$ is
defined by \eqref{e:bj} for all $t<t_\infty$, where $t_\infty=\lim_{k\to\infty}t_k$ is the explosion time. As in
the discrete state space we can extend the state space $E$ by adding the point $-1$ and define $Y(t)=-1$ for
$t\ge t_\infty$. Let
 $\mathbb{P}_x$ be the law of the process $Y$
starting at $Y(0)=x$ and denote by $\mathbb{E}_x$ the expectation with respect to $\mathbb{P}_x$.

\begin{remark}\em
Note that if condition \eqref{e:as1} holds (equivalently $Q(0)=\infty$) then the amount of the gene product
$\{Y(t_k)\}_{k\ge 0}$ at the jump times is a discrete time Markov process with transition probability function
given by
\[
\mathcal{K}(y,B)=\int_B k(x,y)dx,\quad B\in\mathcal{B}((0,\infty)),
\]
where
\begin{equation}\label{e:k}
k(x,y)=e^{Q(y)}\int_{0}^y 1_{(0,x)}(z)h(x-z,z)\frac{\varphi(z)}{\gamma(z)}e^{-Q(z)}dz,\quad x,y>0.
\end{equation}
 If $Q(0)<\infty$ then the random variable $\Delta t_1$ is infinite with positive probability, since we have
for any $x>0$
\[
\Pr(\Delta t_1=\infty |Y(0)=x)=\lim_{t\to\infty}\Pr(\Delta t_1>t |Y(0)=x)=e^{Q(x)-Q(0)}>0,
\]
which then forces the process $Y(t,\omega)$ starting form $Y(0,\omega)=x$ to be $\pi_t(x)$ for all $t$, if
$\omega$ is such that $\Delta t_1(\omega)=\infty$.
\end{remark}

In what follows we assume that \eqref{e:as} and \eqref{e:as1} hold. We rewrite equation \eqref{eq:continuous} as
an abstract Cauchy problem in $L^1$
\begin{equation}\label{e:acpd}
\frac{du}{dt}=\C u,\quad u(0)=v,
\end{equation}
where the operator
\begin{equation}\label{d:C}
\C u(x)=\dfrac{d( \gamma(x)u(x))}{d x}-\varphi(x)u(x)+ \int_0^x\varphi(y)u(t,y)h(x-y,y)dy
\end{equation}
is defined on the domain
\begin{equation}\label{d:dA0n}
\mathcal{D}(\C )=\{u\in L^1: \gamma u\in \mathrm{AC},\; (\gamma u)'\in L^1, \;\lim_{x\uparrow
\infty}(\gamma(x)u(x))=0,\; \varphi u\in L^1\},
\end{equation}
and $ \gamma u\in \mathrm{AC}$ means that the function $x\mapsto \gamma(x)u(x)$ is absolutely continuous.
From~\cite{mackeytyran08,tyran09} it follows that there is a substochastic semigroup $\{P(t)\}_{t\ge 0}$ on
$L^1$ such that for each initial density $v\in \mathcal{D}(\C )$ equation \eqref{e:acpd} has a nonnegative
solution $u(t)$ which is given by $u(t)=P(t)v$ for $t\ge 0$ and
\begin{equation}\label{eq:PB}
\int_0^\infty \mathbb{P}_x(Y(t)\in B, t<t_\infty) v(x)dx=\int_B P(t)v(x)dx
\end{equation}
for all Borel subsets $B$ of $(0,\infty)$. The semigroup $\{P(t)\}_{t\ge 0}$ is stochastic if and only if its
generator $(C,D(C))$ is the closure of the operator $(\C ,\mathcal{D}(\C ))$.

We first study the fixed points of the semigroup, showing that $\{P(t)\}_{t\ge 0}$ has no more that one
invariant density through
\begin{proposition}\label{p:uid}
The substochastic semigroup $\{P(t)\}_{t\ge 0}$ can have at most one invariant density.
\end{proposition}
\begin{proof}
Recall that $u^*$ is an invariant density for the semigroup $\{P(t)\}_{t\ge0}$ if and only if it is an invariant
density for the resolvent operator
\[
Rv:=R(1,C)v=\int_0^\infty e^{-t}P(t)vdt.
\]
The operator $R$ is substochastic and it satisfies $ Rv\ge R_1v $ for any nonnegative $v\in L^1$ (see
\cite{mackeytyran08}), where
\[
R_1v(x)=\frac{1}{\gamma(x)}\int_{x}^\infty v(y) e^{Q(y)-Q(x)+\int_{x}^{y}\frac{1}{\gamma(z)}dz}dy,\quad x>0.
\]
Note that $R_1$ is the resolvent operator $R(1,A)$ of a substochastic semigroup  $\{S(t)\}_{t\ge 0}$ with
generator
\[
Au(x)=\dfrac{d( \gamma(x)u(x))}{d x}-\varphi(x)u(x), \quad u\in \mathcal{D}(\C ).
\]
Since for any two nonnegative and nonzero $v_1,v_2\in L^1$ we can find $c(v_i)>0$ such that
\[
\int_{c(v_i)}^\infty v_i(y)dy>0,\quad i=1,2,
\]
we obtain $Rv_i(x)>0$ for all $x<\min\{c(v_1),c(v_2)\}$, $i=1,2$. Now suppose that $u_1,u_2$ are densities such
that $u=u_1-u_2$ is nonzero. Then both $u^+=\max\{0,u\}$ and $u^{-}=\max\{0,-u\}$ are nonnegative and nonzero.
Thus,  $R(u^+)(x)>0$ and $R(u^-)(x)>0$ for $x<c$ and some $c>0$. We have
\[
|Ru(x)|=|R(u^+)(x)-R(u^{-})(x)|\le R(u^+)(x)+R(u^{-})(x)=R(|u|)(x),
\]
thus the inequality is strict on a set of positive measure, which implies that if $u_1-u_2\neq 0$ then
\[
\|Ru_1-Ru_2\|_1<\|R|u_1-u_2|\|_1\le \|u_1-u_2\|_1.
\]
Consequently, the operator $R$ can have at most one invariant density.
\end{proof}

Let $K$ be the transition operator on $L^1$ given by
\begin{equation}\label{e:oK}
Kv(x)=\int_{0}^\infty k(x,y)v(y)dy,\quad v\in L^1,
\end{equation}
where the kernel $k$ is as in \eqref{e:k}. Observe that
\begin{equation}\label{e:K}
Kv(x)=\int_{0}^{x} h(x-z,z)\frac{\varphi(z)}{\gamma(z)}e^{-Q(z)}\int_{z}^\infty v(y)e^{Q(y)}dydz.
\end{equation}
A mild condition on the transition operator $K$, in conjunction with Theorems 3.6 and 5.2 of~\cite{tyran09}, has
interesting consequences for $\{P(t)\}_{t\ge 0}$ as contained in the following result.
\begin{proposition}\label{p:2}
If the transition operator $K$ is mean ergodic, i.e. for any $v\in L^1$, $v\ge 0$ the sequence
\begin{equation*}
\frac{1}{n}\sum_{j=0}^{n-1} K^j v
\end{equation*}
is convergent in $L^1$, then the semigroup $\{P(t)\}_{t\ge 0}$ is stochastic.
\end{proposition}

In particular, if $K$ has a strictly positive fixed point, i.e. there is $v^*$ such that $Kv^*=v^*$ and $v^*>0$
a.e., then $K$ is mean ergodic~\cite{kornfeldlin00}. Note that a mean ergodic stochastic operator has a nonzero
fixed point.

We now describe invariant densities for the semigroup $\{P(t)\}_{t\ge 0}$ with the help of fixed points of the
operator~$K$.
\begin{theorem}\label{p:id}
Let
\begin{equation*}%\label{eq:tail}
\overline{H}(x,y)=\int_{x}^\infty h(z,y)dz,\quad x>0.
\end{equation*}
Suppose that there is a nonnegative solution $u^*$ of the equation
\begin{equation}\label{e:fpC}
\gamma(x)u^*(x)=\int_{0}^x \overline{H}(x-y,y)\varphi(y)u^*(y)dy
\end{equation}
such that $\varphi u^*\in L^1$. Then  the function
\begin{equation}\label{e:fpK}
v^*(x)=\int_{0}^x h(x-y,y)\varphi(y)u^*(y)dy
\end{equation}
is a fixed point of the operator $K$ in $L^1$, where $K$ is as in~\eqref{e:K}. Moreover, if  $u^*\in L^1$ then
$u^*\in \mathcal{D}(\C )$ and $\C (u^*)=0$, where $\C$ is as in \eqref{d:C}.

Conversely, if the operator $K$ has a nonnegative fixed point $v^*\in L^1$ then the function
\begin{equation}\label{e:uv}
u^*(x):=\frac{1}{\gamma(x)}\int_x^\infty e^{Q(y)-Q(x)}v^*(y)dy
\end{equation}
is a solution of \eqref{e:fpC} and $\varphi u^*\in L^1$.
\end{theorem}
\begin{proof}
Let $u^*$ be a solution of \eqref{e:fpC} such that $\varphi u^*\in L^1$. Since
\[
\lim_{x\to\infty}1_{[y,\infty)}(x) \overline{H}(x-y,y)=0
\]
for each $y$ and $0\le 1_{[y,\infty)}(x) \overline{H}(x-y,y)\le 1$ for all $x,y$, we obtain
\[
\lim_{x\to\infty}\gamma(x)u^*(x)=\lim_{x\to\infty}\int_{0}^\infty 1_{[y,\infty)}(x)
\overline{H}(x-y,y)\varphi(y)u^*(y)dy=0,
\]
by the Lebesgue's dominated convergence theorem. Similarly, we conclude that
\[
\lim_{x\to 0}\gamma(x)u^*(x)=0.
\]
We have
\[
\int_{0}^x \overline{H}(x-y,y)\varphi(y)u^*(y)dy=\int_{0}^x \varphi(y)u^*(y)dy- \int_0^x\int_0^{x-y}h(z,y)dz
\varphi(y) u^*(y)dy.
\]
Thus, $\gamma u^*\in \mathrm{AC}$ and
\begin{equation}\label{e:C0}
\frac{d}{dx}(\gamma(x)u^*(x))=\varphi(x)u^*(x)-\int_0^x h(x-y,y)\varphi(y) u^*(y)dy.
\end{equation}
The functions $\varphi u^*$ and $v^*$ are integrable.  Consequently, if $u^*\in L^1$ then $u^*\in\mathcal{D}(\C
)$ and $\C(u^*)=0$. Since
\[
v^*(x)=\varphi(x)u^*(x)-\frac{d}{dx}(\gamma(x)u^*(x))=-e^{-Q(x)}\frac{d}{dx}(e^{Q(x)}\gamma(x)u^*(x)),
\]
we obtain
\[
\int_{z}^\infty v^*(x)e^{Q(x)}dx=-\int_{z}^\infty
\frac{d}{dx}(e^{Q(x)}\gamma(x)u^*(x))dx=e^{Q(z)}\gamma(z)u^*(z),
\]
which shows that $Kv^*(y)=v^*(y)$, by~\eqref{e:K}.

We now turn to the converse part. Suppose that $u^*$ is as in \eqref{e:uv}, where $v^*$ is a fixed point of $K$.
Since $Q$ is non-increasing and $v^*$ is integrable, we see that
\[
\lim_{x\to\infty}\gamma(x)u^*(x)=0
\]
and that $\varphi u^*\in L^1$. It is easily seen that $u^*$ satisfies equation \eqref{e:C0}. Integrating
equation~\eqref{e:C0} with respect to $x$ from $z$ to $\infty$ leads to
\[
\int_{z}^\infty\frac{d}{dx}(\gamma(x)u^*(x))dx=\int_{z}^\infty\varphi(x)u^*(x)dx-\int_{z}^\infty \int_0^x
h(x-y,y)\varphi(y) u^*(y)dydx
\]
and changing the order of integration in the last integral gives
\[
\begin{split}
\int_{z}^\infty \int_0^x h(x-y,y)\varphi(y) u^*(y)dydx&=\int_{0}^z \int_z^\infty  h(x-y,y)\varphi(y) u^*(y)dxdy\\
&\quad +\int_{z}^\infty \int_y^\infty  h(x-y,y)\varphi(y) u^*(y)dxdy.
\end{split}.
\]
We have
\[
\int_{0}^z \int_z^\infty  h(x-y,y)\varphi(y) u^*(y)dxdy = \int_{0}^z   \overline{H}(z-y,y)\varphi(y) u^*(y)dy
\]
and
\[
\int_{z}^\infty \int_y^\infty  h(x-y,y)\varphi(y) u^*(y)dxdy
%= \int_{z}^\infty \int_0^\infty  h(x,y) dx\varphi(y) u^*(y)dy
=\int_{z}^\infty \varphi(y) u^*(y)dy.
\]
Combining these we conclude that $u^*$ satisfies \eqref{e:fpC}.
\end{proof}

The following theorem guarantees that $\{P(t)\}_{t\ge 0}$ is stochastic and its strong convergence to a unique
stationary density $u^*$ that is given explicitly.

\begin{theorem}\label{thm:4}
Suppose that the operator $K$ as in \eqref{e:K} has an invariant density $v^*>0$ a.e. and let
\[
c:=\int_{0}^\infty \frac{1}{\gamma(x)}\int_x^\infty e^{Q(y)-Q(x)}v^*(y)dydx<\infty.
\]
Then the semigroup $\{P(t)\}_{t\ge 0}$ is stochastic and for each initial density
 $v$ we have
\begin{equation*}
\lim_{t\to\infty}\|P(t)v-u^{*}\|_1=0,
\end{equation*}
where
\[
u^*(x)=\frac{1}{c\gamma(x)}\int_x^\infty e^{Q(y)-Q(x)}v^*(y)dy
\]
is the unique stationary density of $\{P(t)\}_{t\ge 0}$.
\end{theorem}
\begin{proof} By Proposition~\ref{p:2}, the semigroup $\{P(t)\}_{t\ge0}$ is stochastic. From Theorem~\ref{p:id} it
follows that $u^*\in\mathcal{D}(\C )$ and $\C(u^*)=0$. Thus, $u^*$ is an invariant density for the stochastic
semigroup $\{P(t)\}_{t\ge 0}$ and it is unique, by Proposition~\ref{p:uid}. Since $v^*(x)>0$ for a.e. $x>0$, we
conclude that $u^*(x)>0$ for all $x>0$. From assumptions~\eqref{e:as} and~\eqref{e:as1} it follows that there is
a $\delta_0$ such that $\varphi(y)>0$ for $y\in (0,\delta_0)$. This and \eqref{e:fpK} imply that
\[
\int_{0}^\infty \int_{0}^\infty p(x,y)\varphi(y) dy dx>0,\quad \text{where}\quad p(x,y)=1_{(0,x)}(y) h(x-y,y).
\]
Consequently, we can find $t>0$ such that the operator $P(t)$ is partially integral \cite{mackeytyran08} and
the result follows from Theorem~\ref{thm:pr00}.
\end{proof}

We conclude this section with sufficient conditions for mean ergodicity of the transition operator $K$.
\begin{proposition}\label{t:eid} Let $K$ be a transition operator $K$ with a bounded kernel $k$.
Suppose that there exist a nonnegative measurable function $V\colon (0,\infty)\to [0,\infty)$ which is bounded
on bonded subsets of $(0,\infty)$ and constants $a,d>0$  such that
\begin{equation}\label{e:FL}
\int_{0}^\infty V(x)k(x,y)dx\le V(y)-1+a1_{(0,d)}(y),\quad y>0.
\end{equation}
Then the operator $K$ is mean ergodic on $L^1$.
\end{proposition}
\begin{proof}
Let $Z_n, n\ge 0,$ be a Markov chain with stochastic kernel $\mathcal{K}$ given by
\[
\mathcal{K}(y,B)=\int_{B}k(x,y)dx,\quad y>0, B\in \mathcal{B}((0,\infty)).
\]
Recall that a probability measure $\mu$ is invariant for the chain if and only if  the measure $\mu$ satisfies
the equation
\[
\mu(B)=\int_{0}^{\infty} \mathcal{K}(y,B)\mu(dy)
\]
for all Borel measurable sets $B$. We have
\[
\mu(B)=\int_{B}\int_{0}^\infty k(x,y)\mu(dy)dx.
\] Thus each invariant probability measure is absolutely continuous with
respect to the Lebesgue measure on $(0,\infty)$. Since $K$ is the transition operator corresponding to
$\mathcal{K}$, we have
\[
\int_B K^jv(x)dx=\int_{0}^\infty \mathcal{K}^j(y,B)v(y)dy,\quad B\in \mathcal{B}((0,\infty)),
\]
where $\mathcal{K}^1(y,B)=\mathcal{K}(y,B)$ and
\begin{equation*}
\mathcal{K}^{j}(y,B)=\int_0^\infty \mathcal{K}^{j-1}(z,B)\mathcal{K}(y,dz),\quad y>0,j\ge 2.
\end{equation*}
From Theorem 1 and Lemma 1 of \cite{tweedie01} it follows that there exist a finite number of invariant
probability measures $\mu_1,\ldots,\mu_N$ and a finite number of nonnegative functions $L_1,\ldots L_N$ such
that $\sum_{i=1}^NL_i(y)=1$ and
\begin{equation}\label{k:me}
\frac{1}{n}\sum_{j=1}^{n}\mathcal{K}^j(y,B)\to \sum_{i=1}^NL_i(y)\mu_i(B)
\end{equation}
for all $y$ and all Borel sets $B$. Let $v_1,\ldots,v_N$ be the densities of the invariant measures
$\mu_1,\ldots,\mu_N$.  Now let $v\in L^1$. From~\eqref{k:me} and the Lebesgue dominated convergence theorem it
follows that
\[
\lim_{n\to\infty}\int_B \frac{1}{n}\sum_{j=1}^nK^jv(x)dx=\int_{B}\sum_{i=1}^N\int_{0}^{\infty}L_i(y)v(y)dy
v_i(x)dx,
\]
for all Borel $B$. Moreover, the sequence $\frac{1}{n}\sum_{j=1}^nK^jv$ is bounded in $L^1$. Thus, it is weakly
convergent in $L^1$ and, by the mean ergodic theorem, it converges in $L^1$.  % Consequently, we have
%\[
%\lim_{n\to\infty}\int_0^\infty g(y) \frac{1}{n}\sum_{j=1}^nP^jv(y)dy=\int_0^\infty
%g(y)\sum_{i=1}^N\int_{0}^{\infty}L_i(x)v(x)dx v_i(y)dy,
%\]
%for all simple functions $g\in L^\infty$.
\end{proof}

We now apply the last result to our transition operator $K$.

\begin{corollary}\label{cor:existinvK} Let $K$ be the transition operator as in \eqref{e:K} with  bounded $h$.
Suppose that the function
\[
m_1(y)=\int_{0}^\infty xh(x,y)dx,\quad y>0,
\]
is bounded on bounded subsets of $(0,\infty)$. If
\begin{equation}\label{e:eid}
\limsup_{y\to \infty}e^{Q(y)} \int_{0}^y \left(m_1(z)\frac{\varphi(z)}{\gamma(z)}-1\right)e^{-Q(z)}dz<0,
%\lim_{y\to\infty}ye^{Q(y)}=+\infty\quad \text{and}\quad
%\limsup_{y\to\infty}\frac{\varphi(y)m_1(y)}{\gamma(y)}<1,
\end{equation}
%and either $Q(\infty)=0$ or
%\[
%\liminf_{x\to\infty}m_1(y)>0,
%\]
then the operator $K$  is mean ergodic.
\end{corollary}
\begin{proof} Since $K$ has kernel $k$ given by \eqref{e:k}, we obtain
\[
k(x,y)\le c_1 e^{Q(y)}\int_{0}^y \frac{\varphi(z)}{\gamma(z)}e^{-Q(z)}dz=c_1,\quad x,y>0,
\]
where $c_1$ is the upper bound for $h$. By Proposition~\ref{t:eid} it is sufficient to check that the function
$V(x)=x$, up to a multiplicative constant, satisfies condition \eqref{e:FL}. We have
\[
\int_{z}^\infty V(x)h(x-z,z)dx=m_1(z)+z,\quad z>0.
\]
Thus
\[
\int_{0}^\infty V(x)k(x,y)dx=e^{Q(y)}\int_0^y (m_1(z)+z)\frac{\varphi(z)}{\gamma(z)}e^{-Q(z)}dz
\]
for all $y>0$. Since
\[
y=e^{Q(y)}\int_{0}^y z \frac{\varphi(z)}{\gamma(z)}e^{-Q(z)}dz+e^{Q(y)}\int_{0}^y e^{-Q(z)}dz,
\]
we obtain
\[
\int_{0}^\infty V(x)k(x,y)dx-V(y)=e^{Q(y)}\int_0^y \left(m_1(z)\frac{\varphi(z)}{\gamma(z)}-1\right)e^{-Q(z)}dz,
\]
which is a bounded function on sets of the form $(0,d)$. %Let $y_0>0$ and $\delta>0$ be such that
%\[
%m_1(y)\frac{\varphi(y)}{\gamma(y)}-1\le -\delta \quad \text{for}\quad y\ge y_0.
%\]
%%Recall that $Q$ is a non-increasing function. %Then
%%\[
%%W(x)\le
%%\]
%%Suppose first that $Q(\infty)=0$.
%Then
%\[
%e^{Q(x)} \int_{x_0}^x \left(m_1(y)\frac{\varphi(y)}{\gamma(y)}-1\right)e^{-Q(y)}dy\le -a e^{Q(x)-Q(x_0)}(x-x_0)
%\] for all $x\ge x_0\ge y_0$,  thus condition \eqref{e:FL} and the result follows.
%  in this case. Now if  $Q(\infty)=-\infty$, then
% \[
% \limsup_{x\to \infty}e^{Q(x)} \int_{0}^x \left(m_1(y)\frac{\varphi(y)}{\gamma(y)}-1\right)e^{-Q(y)}dy<0,
% \]
% which completes the proof. {\bf MTK does not know how to get the last inequality.}
\end{proof}

\begin{remark}\em
Observe that if $Q(\infty)=0$ and
\[
\limsup_{z\to\infty}\frac{m_1(z)\varphi(z)}{\gamma(z)}<1
\]
then condition \eqref{e:eid} holds, since we can find $z_0>0$ and $\delta>0$ such that
\[
m_1(z)\frac{\varphi(z)}{\gamma(z)}-1\le -\delta \quad \text{for}\quad z\ge z_0,
\]
which implies that
\[
e^{Q(y)} \int_{y_0}^y \left(m_1(z)\frac{\varphi(z)}{\gamma(z)}-1\right)e^{-Q(z)}dz\le -a e^{Q(y)-Q(y_0)}(y-y_0)
\] for all $y\ge y_0\ge z_0$ with the right-hand side going to $-\infty$.

If $Q(\infty)=-\infty$ and
\[
\limsup_{z\to\infty}\left(m_1(z)-\frac{\gamma(z)}{\varphi(z)}\right)<0
\]
then condition \eqref{e:eid} holds as well  by d'Hospital's rule.
\end{remark}

\subsection{Exponentially distributed bursts}\label{ssec:exp-cont}
Experimental findings in populations of cells indicate that the burst size is often exponentially distributed \cite{Xie2008} so we now consider
\begin{equation}\label{c:exp}
h(x,y)=\frac{1}{b}e^{-x/b},\quad x,y>0,
\end{equation}
where $b>0$. The operator $K$ as defined in \eqref{e:K} then takes the form
\[
Kv(x)=\int_{0}^x \frac{1}{b}e^{-(x-z)/b}\frac{\varphi(z)}{\gamma(z)}e^{-Q(z)}\int_{z}^\infty v^*(y)e^{Q(y)}dydz.
\]
Note that the integrable function
\[
v^*(x)=e^{-x/b -Q(x)}
\]
 is a fixed point of the operator $K$, since
\[
Kv^*(x)=e^{-x/b}\int_{0}^x \frac{\varphi(z)}{\gamma(z)}e^{-Q(z)}dz=e^{-x/b-Q(x)}.
\]
Again, an explicit stationary solution was recently obtained in \cite{Aquino11}, and we establish convergence to this stationary state with the following result.

\begin{corollary}\label{cor:1} Assume  that conditions~\eqref{e:as} and \eqref{e:as1} hold and that
 $h$ is exponential as in \eqref{c:exp} with $b>0$.
Suppose that
\begin{equation}\label{e:int}
c:=\int_{0}^{\infty}\frac{1}{\gamma(x)}e^{-x/b -Q(x)}dx<\infty, \quad \int_{0}^{\infty}e^{-x/b -Q(x)}dx<\infty.
\end{equation}
Then the semigroup $\{P(t)\}_{t\ge 0}$ is stochastic and for each initial density
 $v$ we have
\begin{equation*}
\lim_{t\to\infty}\|P(t)v-u^{*}\|_1=0,
\end{equation*}
where
\begin{equation}\label{eq_analytic_density}
u_{*}(x)=\frac{1}{c\gamma(x)}e^{-x/b -Q(x)}
\end{equation}
is the unique stationary density of $\{P(t)\}_{t\ge 0}$.
\end{corollary}

\begin{remark}\em
 Note that if $Q(0)=\infty$ and
\[
\lim_{x\to\infty}\frac{\varphi(x)}{\gamma(x)}<\frac{1}{b},
\]
then the function  $x\mapsto e^{-x/b-Q(x)}$ is integrable on $(0,\infty)$. If, additionally,
\[
\liminf_{x\to\infty}\gamma(x)>0, \quad \lim_{x\to 0}\frac{e^{-Q(x)}}{\gamma(x)^r}<\infty, \quad \text{and}\quad
\int_0^\delta \gamma(x)^{r-1}dx<\infty
\]
for some $\delta,r>0$, then condition~\eqref{e:int} holds.  Furthermore, if it should happen that $b, \gamma(x)$
and $u^*(x)$ are known or can be approximated from data, then it is possible to estimate $\varphi(x)$ from
\begin{equation}\label{eq:pbinverse}
 \varphi(x)=\frac{1}{b}\gamma(x) +\frac{(\gamma(x)u^*(x))'}{u^*(x)}.
\end{equation}
Finally, note that if $\varphi$ is assumed to be bounded, then $u^*$ has an exponential tail, from which we can
deduce the parameter $b$.
\end{remark}
%%%%%%%%%%%%%%%%%
\begin{remark}[Bifurcation in the continuous case]\label{rem:bif_cont}\em
As in the discrete formulation of the model we can use relation \eqref{eq:pbinverse} to derive bifurcation properties of the stationary density  as a function of the relevant parameters. Namely, the number of extrema are linked to the number of solutions of %(if this expression has a sense)
$$\varphi(x)=\frac{\gamma(x)}{b}+\gamma'(x).$$
\end{remark}
%%%%%%%%%%%%%%%%%

In all examples below, we consider a linear degradation function,  $\gamma(x)=\gamma x$ with $\gamma>0$.
%\begin{example} {\color{red}{Marta, this example seems to simply be a special case of the following example and I'm really unsure why it is here! Is it obvious to you?}}
%Consider the function $\varphi$ of the form of a monotone decreasing Hill function
%\[
%\varphi(x)=\frac{\lambda }{1+K_1 x^N},
%\]
%where $\lambda,K_1,N>0$. Then
%\[
%Q(x)=-\frac{\lambda}{\gamma}\log(x)+\frac{\lambda}{\gamma N}\log (1+K_1x^{N})
%\]
%and the stationary density is given by
%\[
%u^*(x)=(c\gamma)^{-1}e^{-x/b}x^{\lambda/\gamma -1}(1+K_1x^N)^{-\lambda/(\gamma N)}.
%\]
%\end{example}

\begin{example}\em  Consider the function $\varphi$ of the form %\cite{Mackey2011}
\begin{equation}\label{eq:genHill}
\varphi(x)=\lambda\frac{1+\Theta x^N}{\Lambda+\Delta x^N}
=\frac{\lambda}{\Delta}+\lambda\left(1-\frac{\Lambda}{\Delta}\right)\frac{1}{\Lambda+\Delta x^N},
\end{equation}
where $\lambda,\Lambda,\Delta, N$ are positive constants and $\Theta\ge 0$. Then
\[
Q(x)=c_1-\frac{\lambda}{\gamma\Lambda}\log(x)+\frac{\lambda}{N\Delta\gamma
}\left(\frac{\Delta}{\Lambda}-\Theta\right)\log (\Lambda+\Delta x^{N}),
\]
where $c_1$ is a constant. The stationary density is given by
\begin{equation}
\label{eq:solcont1} u^{*}(x)=(c\gamma)^{-1}e^{-x/b} x^{\lambda(\gamma\Lambda)^{-1} -1 } (\Lambda + \Delta x^N)^\theta,
\end{equation}
where %$c$ is a normalizing constant and
\begin{equation}\label{eq:theta}
\theta=\frac{\lambda}{N\Delta \gamma}\left(\Theta-\frac{\Delta}{\Lambda}\right).
\end{equation}
This solution has been extensively studied in terms of numbers of maxima (P-bifurcation) in \cite{Mackey2011}
when $\Theta=1$. When $\Theta=\Delta=\Lambda=1$ the density $u^{*}$ is that of a gamma distribution, as
obtained in \cite{Friedman2006}.
\end{example}

\begin{example}\em
Consider the case of linear regulation with the function $\varphi$ of the form
$$\varphi(x)=\lambda_0+\lambda x, $$
where $\lambda_0,\lambda$ are nonnegative constants. If
\begin{eqnarray} \label{cond:cont1} \frac{1}{b}>
\frac{\lambda}{\gamma}\quad \text{and}\quad \lambda_0>0,
\end{eqnarray}
then $u^*$ is integrable and is given by the gamma distribution
\begin{equation}
\label{eq:solcont2}
u^{*}(x)=\frac{1}{\Gamma(\lambda_0/\gamma)}\left(\frac{1}{b}-\frac{\lambda}{\gamma}\right)^{\lambda_0/\gamma}
x^{\frac{\lambda_0}{\gamma}-1}e^{-(\frac{1}{b}-\frac{\lambda}{\gamma})x},
\end{equation}
which is a continuous approximation of the negative binomial distribution previously obtained, as in
\cite{Shahrezaei2008}.
\end{example}

\subsection{Other examples}\label{ssec:separable}

In this subsection we consider some more exactly solvable examples. The class of examples we provide generalizes
the exponentially distributed case of $h$. Let $\nu(y)$ be a positive, decreasing, and absolutely continuous
function on $(0,\infty)$ such that $\nu(y)\to 0$ as $y\to\infty$. Consider the function
\begin{equation}\label{e:sep}
h(x,y)=-\frac{\nu'(x+y)}{\nu(y)}, \quad y,z>0.
\end{equation}
Then for each $y$ the function $x\mapsto h(x,y)$ is a density and
\[
h(x-y,y)=-\frac{\nu'(x)}{\nu(y)}, \quad x>y.
\]
The operator $K$ can be thus rewritten as
\[
Kv(x)=-\int_0^x \frac{\nu'(x)}{\nu(z)}\frac{\varphi(z)}{\gamma(z)}e^{-Q(z)}\int_{z}^\infty v(y)e^{Q(y)}dydz.
\]
It is easily seen that if the function
\[
v^*(x)=-\nu'(x)e^{-Q(x)}
\]
is integrable then $Kv^*(x)=v^*(x)$, thus we obtain the following.
\begin{corollary}\label{cor:4} Let $h$ be as in \eqref{e:sep}. Suppose that
\[
c:=\int_0^\infty\frac{\nu(x)}{\gamma(x)}e^{-Q(x)}dx<\infty\quad\text{and}\quad -\int_{0}^\infty
\nu'(x)e^{-Q(x)}dx<\infty.
\]
Then the semigroup $\{P(t)\}_{t\ge 0}$ is stochastic and for each initial density
 $v$ we have
\begin{equation*}
\lim_{t\to\infty}\|P(t)v-u^{*}\|_1=0,
\end{equation*}
where
\begin{equation*}
u_{*}(x)=\frac{\nu(x)}{c\gamma(x)}e^{-Q(x)}
\end{equation*}
is the unique stationary density of $\{P(t)\}_{t\ge 0}$.
\end{corollary}
%%%%%%%%%%%%
\begin{remark}[Bifurcation in the continuous case--again]\em
As before (see Remark \ref{rem:bif_cont}), the number of extrema are linked to the number of solutions of %(if this expression has a sense)
$$\varphi(x)=-\frac{\nu'(x)}{\nu(x)}+\gamma'(x).$$
%\end{remark}
%%%%%%%%%%%%
%\begin{remark}
 Note that if it should happen that $\nu(x), \gamma(x)$ and $u^*(x)$ are known or can be approximated from data, then it is possible to estimate $\varphi(x)$ from
\begin{equation*}%\label{eq:pbinverse}
 \varphi(x)=-\frac{\nu'(x)}{\nu(x)}\gamma(x) +\frac{(\gamma(x)u^*(x))'}{u^*(x)}.
\end{equation*}
If $m_1(x)=\int_{0}^\infty zh(z,x)dz<\infty$, then only the knowledge of $m_1(x)$ is sufficient as
\begin{equation*}
 -\frac{\nu'(x)}{\nu(x)}=\frac{1+m_1'(x)}{m_1(x)}.
\end{equation*}
\end{remark}
%%%%%%%
In the examples below, we take a linear degradation function $\gamma(x)=\gamma x$, with $\gamma>0$.
\begin{example}\em
Suppose that the function $\nu$ is of the form
\[
\nu(x)=(\alpha+x)^{-\beta}
\]
where $\alpha,\beta >0$ and that the function $\varphi$ is of the form \eqref{eq:genHill}. If
$$\beta >\frac{\lambda\Theta}{\gamma\Delta }+1$$
then the assumptions of Corollary~\ref{cor:4} are satisfied and  the stationary density $u^*$ %is integrable and
is given by
\begin{equation*}
u^{*}(x)=\frac{1}{c\gamma}(\alpha +x)^{-\beta} x^{\lambda(\gamma\Lambda)^{-1} -1 } (\Lambda + \Delta x^N)^\theta,
\end{equation*}
where %$c$ is a normalizing constant and
$\theta$ is as in \eqref{eq:theta}.
\end{example}

\begin{example}\em
Suppose that the function $\nu$ is of the form
\[
\nu(x)=e^{-(\alpha x+ \beta x^2)},
\]
where $\alpha, \beta >0$. Consider the case of linear regulation with the function $\varphi$ of the form
$$\varphi(x)=\lambda_0+\lambda_1 x, $$
where $\lambda_0,\lambda_1$ are nonnegative constants. If
\begin{eqnarray*}
\lambda_0>0,
\end{eqnarray*}
then $u^*$ is integrable and is given by
\begin{equation*}
u^{*}(x)=\frac{1}{c\gamma}x^{\frac{\lambda_0}{\gamma}-1}e^{-(\alpha -\frac{\lambda_1}{\gamma})x- \beta x^2}.
\end{equation*}
Consider the case of quadratic regulation with the function $\varphi$ of the form
$$\varphi(x)=\lambda_0+\lambda_1 x+\lambda_2 x^2, $$
where $\lambda_0,\lambda_1,\lambda_2$ are nonnegative constants.
If
\begin{eqnarray*}
\beta >\frac{\lambda_2}{2\gamma} \quad \text{and} \quad \lambda_0>0,
\end{eqnarray*}
then $u^*$ is integrable and is given by
\begin{equation*}
u^{*}(x)=\frac{1}{c\gamma}x^{\frac{\lambda_0}{\gamma}-1}e^{-(\alpha -\frac{\lambda_1}{\gamma}) x-( \beta
-\frac{\lambda_2}{2\gamma})x^2}.
\end{equation*}
%Note that in this case $m_1(x)=\frac{\sqrt{2}}{\alpha \sqrt{\beta }}e^{\beta x^2/2}erfc(\sqrt{\beta /2}x)\sim \frac{2}{\sqrt{\pi}\alpha \beta x}$ as $x\to \infty$.
\end{example}

\begin{example}\em
Suppose that the function $\nu$ is of the form
\[
\nu(x)=(\alpha -x)^\beta ,
\]
where $\alpha, \beta >0$, for all $x<\alpha$, and $\nu(x)=0$ for $x\geq \alpha$. Suppose the function $\varphi$
is given by \eqref{eq:genHill} where $\lambda,\Lambda,\Delta, N$ are positive constants and $\Theta\ge 0$. Then
the stationary density $u^*$ is integrable and is given by, for all $x< \alpha $,
\begin{equation*}
u^{*}(x)=\frac{1}{c\gamma}(\alpha -x)^{\beta } x^{\lambda(\gamma\Lambda)^{-1} -1 } (\Lambda + \Delta x^N)^\theta,
\end{equation*}
where %$c$ is a normalizing constant and
$\theta$ is as in \eqref{eq:theta}. Convergence is obtained in the state space %can be obtained by considering
%the state space to be
$(0,\alpha)$. %Note that in this case $m_1(x)=\dfrac{\alpha -x}{1+\beta}$.

\end{example}

%%%%%%%%%%%%%%%%%%%%%%%%

\section{Conclusions and summary}\label{sec:conclusions}

In this paper we have presented both a discrete Markov process formulation as well as a continuous model
formulation for bursting gene expression.  Our development of the discrete model formulation in Section
\ref{ssec:gen-discrete} allowed us to prove a very general convergence result in Theorem \ref{t:exst} and then
to use that result to explore a variety of examples in Section \ref{ssec:disc-geo} when the burst amplitude is
geometrically distributed.  In Section \ref{sec:cont} we developed the analogous continuous model for bursting
expression.  Section \ref{ssec:gen-cont} contains the general development with Proposition \ref{p:uid} limiting
the number of invariant densities of the semigroup $\{P(t)\}_{t\ge 0}$, while Proposition \ref{p:2} uses mean
ergodicity of the transition operator $K$ to show that $\{P(t)\}_{t\ge 0}$ is stochastic.  Theorems \ref{p:id}
and \ref{thm:4} give criteria for a unique stationary density $u^*$ of $\{P(t)\}_{t\ge 0}$ and for convergence
to that stationary density.  In Section \ref{ssec:exp-cont} we have used these results in a number of specific
examples when the burst amplitudes are exponentially distributed--a situation often noted experimentally.
Section \ref{ssec:separable} concludes  with an examination of a generalization of the exponential distribution
of burst amplitudes.

%\section{Acknowledgements}
%This research was supported by Natural Sciences and Engineering
%Research Council of Canada (MCM), the State Committee for Scientific Research (Poland) Grant no ~N~N201~608240
%(MT-K), and the Ecole Normale Superieure Lyon (ENS Lyon, France, RV).

%\bibliographystyle{siam}
%\bibliography{bib-markov}
%\end{document}

\end{document}